\documentclass[12pt]{amsart}
\usepackage[margin=2.0cm]{geometry}

\usepackage{amssymb}
\usepackage{amsfonts}
\usepackage{mathrsfs}
\usepackage{cite}
\usepackage{graphicx}
\usepackage{mathtools}

\usepackage{float}

\usepackage{todonotes}

\usepackage{pgfplots}
\pgfplotsset{compat=1.15}
\usepackage{mathrsfs}
\usetikzlibrary{arrows}

\definecolor{ffqqqq}{rgb}{1,0,0}
\definecolor{qqzzqq}{rgb}{0,0.6,0}

\newcommand{\R}{{\mathbb R}}

\newcommand{\half}{\frac{1}{2}}
\newtheorem{theorem}{Theorem}[section]

\newtheorem{lemma}[theorem]{Lemma}
\newtheorem{prop}[theorem]{Proposition}

\newcommand{\creg}{C_{\mathrm{reg}}}
\newcommand{\cc}{\mathfrak{c}}

\DeclareMathOperator{\arccot}{\mathrm{arccot}}
\DeclareMathOperator{\diam}{\mathrm{diam}}

\DeclareMathOperator{\conv}{\mathrm{conv}}
\DeclareMathOperator{\area}{\mathrm{area}}

\DeclareMathOperator{\pos}{pos}
\DeclareMathOperator{\per}{per}

\newcommand{\inte}{{\operatorname{int}}}

\makeatletter
\DeclareFontFamily{U}{tipa}{}
\DeclareFontShape{U}{tipa}{m}{n}{<->tipa10}{}
\newcommand{\arc@char}{{\usefont{U}{tipa}{m}{n}\symbol{62}}}%

\newcommand{\arc}[1]{\mathpalette\arc@arc{#1}}

\newcommand{\arc@arc}[2]{%
  \sbox0{$\m@th#1#2$}%
  \vbox{
    \hbox{\resizebox{\wd0}{\height}{\arc@char}}
    \nointerlineskip
    \box0
  }%
}
\makeatother

\title{The isominwidth problem on the 2-sphere}

\author{Ansgar Freyer}
\thanks{Ansgar Freyer is funded by the Deutsche Forschungsgemeinschaft (DFG, German Research Foundation) - 539867386.}

\author{\'Ad\'am Sagmeister}
\thanks{}

\address{FU Berlin, Fachbereich Mathematik und Informatik, Arnimallee 2, D-14195 Berlin, Germany} \email{a.freyer@fu-berlin.de}
\address{Bolyai Institute, University of Szeged, Aradi vértanúk tere 1, H-6720 Szeged,
Hungary} \email{sagmeister.adam@gmail.com}

\mathtoolsset{showonlyrefs=true}

\begin{document}

\begin{abstract}
P\'al's isominwidth theorem states that for a fixed minimal width, the regular triangle has minimal area. A spherical version of this theorem was proven by Bezdek and Blekherman, if the minimal width is at most $\tfrac \pi 2$. If the width is greater than $\tfrac \pi 2$, the regular triangle no longer minimizes the area at fixed minimal width. We show that the minimizers are instead given by the polar sets of spherical Reuleaux triangles. Moreover, stability versions of the two spherical inequalities are obtained.
\end{abstract}

\maketitle

\section{Introduction}
\label{sec:intro}

The (minimal) width $w(K)$, or sometimes also referred to as thickness, of a convex body $K\subset\R^n$, i.e., of  a compact convex set in $\R^n$, is commonly defined as the minimum distance of two parallel supporting hyperplanes of $K$. 
Back in 1921, P\'al studied the relation between the area and the width of a 2-dimensional convex body $K\subset\R^2$. He proved that if the width $w$ is fixed, the unique (up to isometry) convex body $K\subset\R^2$ of minimal area is the regular triangle whose width is $w$ \cite{Pal}. This beautiful result may be regarded as a dual version of the isodiametric inequality, which states among all convex bodies $K\subset\R^n$ of fixed diameter $D$, the Euclidean $n$-balls of diameter $D$ uniquely maximize the volume \cite[Eq.\  (7.22)]{Sch14}. While the isodiametric inequality holds in all dimensions, a higher-dimensional analog of P\'al's result is not at hand. The problem of determining among the convex bodies $K\subset\R^n$ for fixed width $w$ the ones of minimal volume is known as P\'al's problem, or as the isominwidth problem.

Apart from the Euclidean setting, an isominwidth (or isodiametric) problem may be posed in each kind of geometry that allows for a definition of width (diameter) and volume. For spaces of constant curvature the isodiametric problem has been solved by Schmidt \cite{Sch48,Sch49} and
B\"or\"oczky, Sagmeister \cite{BoS20,BoS23}.
On the 2-sphere $S^2$, the isominwidth inequality has been proven by Bezdek and Blekherman \cite{Bez00} for width $w\leq\frac{\pi}{2}$. The spherical width functional $w(K)$ is defined in Section \ref{subsec:width}.

\begin{theorem}[Bezdek--Blekherman\cite{Bez00}]\label{isominwidth}
Let $K\subset S^2$ be a convex body of width $0<w\leq\frac{\pi}{2}$, and let $T_{w}\subset S^2$ denote an equilateral triangle of width $w$. Then we have
$$
\area\left(K\right)\geq \area\left(T_{w}\right).
$$
\end{theorem}
Moreover, equality holds in the spherical isominwidth problem only if $K$ is congruent with $T_{w}$ (cf.\ Section \ref{sec:pal:smallw}).

Recently various stability results on P\'al's problem in the Euclidean plane have been obtained by Lucardesi and Zucco \cite{LZ24}. Here, we study the stability of the spherical isominwidth problem. We denote the Hausdorff distance of two sets $K,L\subset S^n$ by $\delta\left(K,L\right)$. Our first main result is a stability version of P\'al's inequality on these planes.

\begin{theorem}
    \label{stab:thm}
  Let $T_w\subset S^2$ denote a regular triangle of width $0<w\leq\frac{\pi}{2}$. There exists a constant $\cc_w>0$ depending only on $w$ such that for any $\varepsilon>0$ and any convex body $K\subset S^2$ with $w(K) = w$ and
    $\area(K)\leq\area(T_w) + \varepsilon$ we have $\delta(K,T)\leq\cc_w\varepsilon$ for a certain regular triangle $T$ of width $w$.
\end{theorem}

One of the main difficulties in the study of the isominwidth problem for $n>2$ is the absence of a natural candidate for a minimizer. While for $n=2$, the regular triangle is the unique minimizer of the volume at given width, the same cannot be said about the regular tetrahedron $T\subset\R^3$. Since the width of the regular tetrahedron is attained by two supporting hyperplanes at opposite edges, one might truncate one of the vertices of $T$ slightly and obtain a convex body $K$ strictly contained in $T$ without affecting the width. This leads to the definition of \emph{reduced convex bodies}. A convex body $K$ is called reduced, if there is no convex body $K^\prime \subsetneq K$ with $w(K^\prime) = w(K)$.  Reduced bodies were first introduced by Heil \cite{H78}. Since the volume is a strictly increasing functional on the class of convex bodies, it is clear that extremizers of the isominwidth problem need to be reduced. 

It has been pointed out by Lassak in his survey \cite{La20} that a regular triangle of width $w>\tfrac{\pi}{2}$ is not a reduced body in $S^2$. Therefore, we cannot expect Theorem \ref{isominwidth} to hold in $S^2$ if $w > \tfrac \pi 2$. We use the structure of reduced bodies and their relatives, \emph{bodies of constant width}, to extend Theorem \ref{isominwidth} to the case where $w>\tfrac{\pi}{2}$.

\begin{theorem}
\label{intro:thm:largewidth}
    Let $K\subset S^2$ be a convex body of width $w > \tfrac{\pi}{2}$. Then we have
    \[
    \area(K)\geq\area(U_{\pi-w}^\circ),
    \]
    where $U_{\pi-w}^\circ$ is the convex hull of three spherical caps of radius $w-\tfrac{\pi}{2}$ whose centers form a regular spherical triangle with edge length $\pi -w $. Equality holds if and only if $K$ is congruent with $U_{\pi-w}^\circ$.
\end{theorem}

Note that for $w=\tfrac \pi 2$ the body $U^\circ_{\pi/2}$ coincides with the regular triangle of width $\tfrac \pi 2$, so Theorem \ref{intro:thm:largewidth} is indeed accordant with Theorem \ref{isominwidth}.
In fact, the extremal body $U_{\pi-w}^\circ$ in the above theorem is the spherical polar body of the so-called Reuleaux triangle $U_{\pi-w}$ of width $\pi-w$ (see Section \ref{subsec:polarity}).

We also show the following stability version of Theorem~\ref{intro:thm:largewidth}.

\begin{theorem}
\label{intro:thm:largewidthstab}
    Let $U_{\pi-w}\subset S^2$ denote a Reuleaux triangle of width $\pi-w$ for $w\in\left(\frac{\pi}{2},\pi\right)$. There exists a constant $\cc_w>0$ depending only on $w$ such that for any $\varepsilon>0$ and any convex body $K\subset S^2$ with
    $w(K)=w$ and $\area(K)\leq\area(U_{\pi-w}^\circ) + \varepsilon$ we have $\delta(K,U^\circ)\leq\cc_w\sqrt\varepsilon$ for a certain Reuleaux triangle $U$ of width $\pi-w$.

\end{theorem}

The paper is organized as follows. In Section \ref{secconstant} we provide the necessary material from spherical convex geometry. In Section \ref{sec:pal:smallw} we first revisit the proof of Theorem \ref{isominwidth}. We take a more analytic perspective on the proof that will allow us to identify regular triangles as the unique minimizers in the isominwidth problem if the width is not greater than $\tfrac \pi 2$. Afterwards, we refine this analysis to obtain Theorem \ref{stab:thm}. In Section \ref{sec:pal:largew} we solve the isominwidth problem in the obtuse setting (Theorem \ref{intro:thm:largewidth}) and prove a stability version (Theorem \ref{intro:thm:largewidthstab}). 

\section{Convexity on the sphere}
\label{secconstant}

For a pair of non-antipodal points $x,y\in S^n$, the unique geodesic segment connecting $x$ and $y$ is denoted by $\left[x,y\right]$. The spherical distance of $x$ and $y$ is denoted by $d\left(x,y\right)$ and the diameter of $X$ is denoted by $\diam X$. We call a set $X\subseteq S^n$ \emph{convex}, if it is contained in an open hemisphere and for each pair of points $x,y\in X$, we have $\left[x,y\right]\subseteq X$. Similarly to Euclidean convexity, intersection preserves convexity on the sphere too, so we can define the \emph{convex hull} of a subset $X$ of an open hemisphere as the intersection of all spherically convex sets containing $X$ denoted as $\conv X$. If $X = \{x_1,\dots, x_k\}$ is a finite set, we also write $[x_1,\dots,x_k]=\conv\{x_1,\dots,x_k\}$.  A set $X\subseteq S^n$ is called a \emph{convex body} if it is convex, compact, and its interior is not empty. Spherical caps (balls in $S^n$ with respect to the spherical metric) are denoted as $B\left(x,r\right)$ where $x$ is the center of the ball, while $0<r\leq\pi$ denotes its radius. We note that balls of radius $r$ are convex if and only if $r<\frac{\pi}{2}$. 

\subsection{Width of spherical convex bodies}\label{subsec:width}

A notion of width in $S^n$ is defined with lunes by Lassak\cite{La15}. A \emph{lune} is an intersection of two different closed hemispheres of the unit sphere $S^n$ whose centers are not antipodal. The boundary of a lune consists of two $(n-1)$-dimensional hemispheres that intersect in an $(n-2)$-sphere.
Given a lune $L\subset S^n$, the \emph{breadth} $b(L)$ of $L$ is the geodesic distance of the centers of the two $(n-1)$-dimensional hemispheres bounding $L$. For a spherically convex body $K\subset S^n$ the lune $L$ is a \emph{supporting lune} if $K\subseteq L$ and $K$ intersects both hemispheres bounding $L$. %
The width of a spherically convex body regarding a supporting hyperplane $H$ is the minimal breadth among all the supporting lunes to $K$ such that one of the bounding $\left(n-1\right)$-dimensional hemispheres of these supporting lunes is contained in $H$; we use the notation $w_K\left(H\right)$ for this quantity. Lassak showed \cite{La15}, that the spherical width function is continuous, so for a convex body both the maximal and the minimal width are attained. The \emph{(minimal) width} of a convex body $K\subset S^n$ is
$$
w\left(K\right)=\min\left\{w_K\left(H\right)\colon H\text{ is a supporting hyperplane to }K\right\}.
$$
A convex body $K\subset S^n$ is called \emph{reduced}, if for any convex body $L\subsetneq K$, we have $w\left(L\right)<w\left(K\right)$. A \emph{reduction} of a convex body $K$ is a reduced convex body $L\subseteq K$ with $w\left(L\right)=w\left(K\right)$. It follows from Zorn's Lemma that every convex body admits a reduction. 

In this paper, we say that $K$ is \emph{of constant width}, if its width functional $w_K(H)$ is constant on the set of supporting hyperplanes of $K$. This is also the definition used in Lassak's survey \cite{La20}. It should be mentioned that several other definitions of ``constant width'' are in use for spherical convex bodies. Some of the results that we use in the following are originally stated for different definitions of constant width. Specifically, these are Theorems \ref{thm:leichtweiss}, \ref{thm:araujo} and \ref{BL-stab}.  However, according to Lassak \cite[Section 12]{La20}, most of these notions are equivalent to the definition of constant width that we fixed. The only exception is the defintion proposed by Dekster \cite{Dek95} which is at the heart of Theorem \ref{BL-stab}. Dekster's definition is however equivalent to the definition that we use provided that $w(K)\leq \tfrac\pi 2$. Since Theorem \ref{BL-stab} holds only for such convex bodies we do not need to distinguish the various definitions of constant width in the following.

It is clear that a body of constant width is reduced. Remarkably, the converse also holds, provided that $w(K) \geq \tfrac \pi 2$:

 \begin{theorem}[Lassak, \cite{La201}]
\label{cor:lassak}
Let $K\subset S^n$ be a reduced convex body of width $w\geq \tfrac\pi 2$. Then $K$ is of constant width $w$.
\end{theorem}

\subsection{Polarity on the sphere}\label{subsec:polarity}

We recall that for a set $A\subset S^n$, its \emph{polar set} is defined as
\[
A^\circ = \{u\in S^n \colon\left<a,u\right>\leq 0,~\forall a\in A\}. 
\]
If $A$ is a convex body on $S^n$, then $A^\circ$ is a convex body on $S^n$. In that case, we have $(A^\circ)^\circ = A$. The polarity on $S^n$ can be expressed in terms of the polarity of convex cones: For a set $A\subseteq S^n$ we consider its positive hull 
\[
\pos A = \{\lambda a\colon \lambda\geq 0, a\in A\}\subseteq \R^{n+1}.
\]
That way, $A$ is a convex body, if and only if $\pos A$ is a convex pointed cone. We have $A^\circ = (\pos A)^\star \cap S^n$, where 
$$
(\pos A)^\star = \{y\in\R^{n+1} \colon \langle x,y\rangle \leq 0,~\forall x\in \pos A\}
$$
is the dual cone of $\pos A$.
The following lemma allows to compute the width of a convex body in terms of the diameter of its polar and vice versa.

\begin{lemma}
\label{lemma:luneslines}
Let $u,v\in S^n$, $n\geq 2$, be two distinct points and let
\[
L_{uv} = \{ x\in S^n \colon \left<x,u\right> \leq 0,~\langle x,v\rangle \leq 0\}
\]
be the lune bounded by the hemispheres opposite to $u$ and $v$ respectively. Then, $L_{uv}^\circ = [u,v]$ and we have $d(u,v) = \pi- b(L_{uv})$.
\end{lemma}

\begin{proof}
We have $\pos L_{uv} = \{x\in\R^{n+1}\colon \langle x, u\rangle\leq 0,~\langle x,v\rangle\leq 0\}$. Hence, $(\pos L_{uv})^\star = \pos\{u,v\}$, which implies $L_{uv}^\circ = \pos\{u,v\}\cap S^n = [u,v]_{S^n}$ as desired.

For the second statement, let $u^\prime$ be the center of the $(n-1)$-dimensional hemisphere $H_u$ bounding $L_{uv}$, orthogonal to $u$. Likewise, let $v^\prime$ be the center of the $(n-1)$-dimensional hemisphere $H_v$ bounding $L_{uv}$ which is orthogonal to $v$. Consider the $(n-1)$-dimensional linear space $V\subset \R^{n+1}$ which is orthogonal to $u$ and $v$. We claim that $u^\prime,v^\prime\in V^\perp$, where $V^\perp$ denotes the orthogonal complement of $V$. 

To see this, let $S = S^n\cap V$. This is an $(n-2)$-dimensional subsphere of $S^n$ and it is clear that $S = \partial H_u=\partial H_v$. But since $u^\prime$ is the center of the $(n-1)$-hemisphere $H_u$, we have $d_{S^n}(u^\prime,x) = \tfrac\pi 2$, for all $x\in\partial H_u$. This is means that $u^\prime$ is orthogonal to any $x\in\partial H_u$, i.e., $u^\prime\in V^\perp$. For the same reason, we have $v^\prime \in V^\perp$.

So we have seen that $\{u,v,u^\prime,v^\prime\}$ are contained in the great circle $C=V^\perp\cap S^n$. Even more, it can be checked quickly that the cyclic order of the four points on $S$ is $(u,v,v^\prime,u^\prime)$. Therefore, and since $x$ is by construction orthogonal to $x^\prime$, $x\in\{u,v\}$, we have
\[
2\pi = d(u,v) +\frac{\pi}{2} + d(u^\prime,v^\prime) +\frac{\pi}{2} = d(u,v)+b(L_{uv}) +\pi,
\] 
which concludes the proof.
\end{proof}

In \cite{HW21} Han and Wu showed that the class of bodies of constant width is closed under polarity.

\begin{prop}[Han--Wu \cite{HW21}]
    \label{prop:constwidthpolar}
    Let $K\subsetneq S^n$ be a convex body. Then, $K$ is of constant width $w$, if and only if $K^\circ$ is of constant width $\pi -w$. 
\end{prop}

Together with Theorem \ref{cor:lassak} one immediately obtains the following equivalence that will be crucial for the isominwidth inequality with $w(K)\geq \tfrac \pi 2$.

\begin{lemma}
\label{prop:duality}
    Let $K\subset S^n$ be a convex body and let $w\geq \tfrac \pi 2$. Then $K$ is reduced with $w(K)=w$, if and only if $K^\circ$ is of constant width with $w(K^\circ) = \pi - w$.
\end{lemma}

\begin{proof}
    Suppose that $K$ is reduced with $w(K) = w$. Then $K$ is of constant width $w$ by Theorem \ref{cor:lassak}. According to Proposition \ref{prop:constwidthpolar}, $K^\circ$ is of constant width $\pi-w$.

    For the converse we assume that $K^\circ$ is of constant width $\pi-w$. Proposition \ref{prop:constwidthpolar} yields that $K$ is of constant width $w$. It follows easily that $K$ is reduced.
\end{proof}

\subsection{Inradius of spherical convex bodies}\label{subsec:balls}

For a convex body $K\subset S^n$ we its inradius $r(K)$ is defined as the largest radius $r>0$ such that there exists $p\in S^n$ with $B(p,r)\subseteq K$.
As in the Euclidean setting, the inradius is witnessed by a suitable set of contact points of the inscribed ball and $\partial K$. For a proof, see B\"or\"oczky, Sagmeister\cite{BoS22}.

\begin{lemma}\label{closest-points-on-the-boundary}
If $K\subseteq S^n$ is a convex body and $B$ is its inscribed ball, then for $2\leq k\leq n+1$ there are points $t_1,\ldots,t_k\in\partial K\cap\partial B$ such that $\left[t_1,\ldots,t_k\right]$ is a $\left(k-1\right)$-dimensional simplex and the center $p$ of $B$ is in the relative interior of $\left[t_1,\ldots,t_k\right]$.
\end{lemma}

Using polarity, it is shown in Bezdek, Blekherman \cite{Bez00} that the regular triangle minimizes the inradius among convex bodies on $S^2$ of fixed width $w\leq\tfrac{2\pi}{3}$ (see Blaschke\cite{Bla15} for the Euclidean case). Here we give a direct proof for the case $w\leq \tfrac{\pi}{2}$ that will allow us to characterize regular triangles as the \emph{unique} minimizers of the inradius.

\begin{lemma}\label{Blaschke}
Let $K\subset S^2$ be a convex body of width $0<w\leq\frac{\pi}{2}$. Let $T_w\subset S^2$ be a regular triangle of width $w$. Then,
$$
r(T_w)\leq r(K)
$$
with equality if and only if $K$ is a congruent copy of $T_w$.
\end{lemma}

\begin{proof}
    Suppose that $r(K)\leq r(T_w)$. Let $B$ be the incircle of $K$. Then there exist three (not necessarily distinct) contact points $t_1,t_2,t_3\in \partial K \cap \partial B$ such that $p\in [t_1,t_2,t_3]$ (cf.\ Lemma \ref{closest-points-on-the-boundary}), where $p$ denotes the center of $B$. Let $\ell_i$ be the tangent line to $K$ (and thus to $B$) at $t_i$. Let $T$ be the triangle bounded by the lines $\ell_i$ that contains $K$. Then, $w(T)\geq w(K) = w$ and $r(T) \leq r(K)\leq r(T_w)$. 

    It suffices to show that $T$ is congruent with a regular triangle $T_w$ of width $w$. Indeed, since $w\leq \tfrac \pi 2$, the triangle $T_w$ is reduced. Since $K$ has the same width as $T_w$, it follows that $K=T$ as desired.

    In order to show that $T$ is congruent with $T_w$, we consider the minimal angle $\alpha = \angle(t_i,p,t_j)$. Without loss of generality, we assume that $\alpha$ is attained for $\{i,j\} = \{2,3\}$.
    Clearly, we have $\alpha \leq \tfrac{2\pi}{3}$ and equality holds if and only $T$ is a regular triangle.
    
    Since $\alpha$ is minimal and since $p$ is contained in $[t_1,t_2,t_3]$ all the remaining angles of the form $\angle(t_i,p,t_j)$, $\{i,j\}\neq\{1,2\}$, are at least $\tfrac{\pi}{2}$.
    Hence, the vertex $v_1$ is contained in the half space $\ell_0^+$ defined by the line $\ell_0$ which intersects $[t_1,p]$ orthogonally and $\ell_0^+$ does not contain $t_1$.

    Let $q_1$ be the point of distance $w$ to $t_1$ such that $p\in [t_1,q_1]$. Let $L$ be the lune of breadth $w$ such that one of its bounding 1-hemispheres is contained in $\ell_1$ and such that $t_1$ and $q_1$ are the centers of its two bounding 1-hemispheres. Since the distance of $v_1$ to $p$ is strictly increasing in $\alpha$, we deduce that $d(p_1,v_1)$ is maximal, if and only if $\alpha = \tfrac{2\pi} 3$ in which case $T$ is a regular triangle with $r(T) \leq r(T_w)$. Since $w(T) \geq w$ it we see that $T$ has width $w$ if it is regular. So we obtain $d(p,v_1) \leq d(p,q_1)$ with equality if and only if $T$ is regular. Hence $v_1\in \ell_0^+ \cap B(p,w-r(K))\subset L$ and therefore $T\subset L$. We obtain
    \[
        w\leq w(T) \leq w_{\ell_1}(T) \leq b(L) = w.
    \]
    If $d(p,v_1) < d(p,q_1)$, $L$ would not be a supporting lune of $T$, so the last inequality would be strict, yielding a contradiction. Hence, $d(p,v_1)=d(p,q_1)$ which implies that $T$ is a regular triangle of width $w$ as desired.
\end{proof}

As we will see in Section \ref{sec:pal:largew}, the assumption $w\leq \tfrac{\pi}{2}$ is necessary to characterize the equality cases. While the inequality of Bezdek and Blekherman also holds if $w\in (\tfrac{\pi}{2},\tfrac{2\pi}{3}]$, there are more equality cases than the regular triangles in this setting.

\section{P\'al's problem for \boldmath$w\leq\frac{\pi}{2}$}
\label{sec:pal:smallw}

\subsection{The spherical isominwidth inequality for $w\leq\frac{\pi}{2}$}\label{subsec:pal:smallw}

In this section we recall the proof of the spherical isominwidth inequality Theorem \ref{isominwidth} due to Bezdek and Blekherman in order to obtain a blueprint for the proof of our stability result Theorem \ref{stab:thm}. We will also explain how the equality case can be derived from their argument.

For a convex body $K\subset S^2$ of width $w\in (0,\tfrac \pi 2]$ we have that $r(T_w) \leq r(K) \leq \tfrac w2$, where again $T_w$ denotes a regular triangle of width $w$ (cf.\ Lemma \ref{Blaschke}) and the upper bound trivially follows from the fact that the width of a ball is twice its radius. The first step in the proof of the spherical isominwidth inequality in Bezdek, Blekherman \cite{Bez00} is the following proposition.

\begin{prop}
    \label{prop:capdomain}
    Let $K\subset S^2$ be a convex body of width $w\in (0,\tfrac \pi 2]$. Let $B\subseteq K$ be its incircle. Then there exists $q_1,q_2,q_3\in K$ such that $d(p,q_i) = w-r$ and the caps $\big(\conv(B\cup \{q_i\})\big)\setminus B$ have disjoint interiors.
\end{prop}

This proposition yields a so-called cap domain
\begin{equation}
    \label{eq:capdomaindef}
    C = \conv(B\cup\{q_1,q_2,q_3\})\subseteq K.
\end{equation}
If $r(K) = \tfrac w2$, $C$ degenerates to $B$. It is enough to lower bound the area of $C$. We write $C_i$ for the cap $\big(\conv(B\cup \{q_i\})\big)\setminus B$. Note that all the three caps are congruent. Since they are non-overlapping, one can replace the apexes $q_i$ by new points $\widetilde{q}_i$ on $\partial B(p,w-r)$ such that $\widetilde{q}_1 = q_1$ and $\angle(\widetilde{q}_i,p,\widetilde{q}_1) = \tfrac{2\pi}{3}$ and $\{q_1,q_2,q_3\}$ is in the same clockwise order around $p$ than $\{\widetilde{q}_1, \widetilde{q}_2,\widetilde{q}_3\}$. The resulting regularized cap domain
\begin{equation}
    \label{eq:regcapdomaindef}
    C_{\mathrm{reg}} = \conv(B\cup\{\widetilde{q}_1, \widetilde{q}_2,\widetilde{q}_3\})
\end{equation}
has non-overlapping caps. Hence, $\area(C_{\mathrm{reg}}) = \area(C)$. Note that $C_{\mathrm{reg}}$ depends solely on $w$, which is fixed, and $r$, which varies in $[r(T_w), \tfrac w2]$. 

Bezdek and Blekherman finish their proof of the spherical isominwidth inequality by showing that $\area(C_{\mathrm {reg}}) \geq \area(T_w)$ via a geometric argument. For the purpose of our stability result and, in particular, for the characterization of the equality case, we prove the following lemma.

\begin{lemma}
    \label{lemma:hexagon_formular}
    Let $w\in (0,\tfrac \pi 2]$ and $p\in S^2$ be fixed. For $\eta\in [0,\tfrac w2-r(T_w)]$, let $C_{\mathrm{reg}}(\eta)$ denote the convex hull of $B(p,r(T_w) + \eta)$ and three points $\widetilde{q}_i\in \partial B(p, w-r(T_w)-\eta)$ such that $\angle (\widetilde{q}_i, p,\widetilde{q}_j) = \tfrac{2\pi}{3}$, $i,j\in\{1,2,3\}$, $i\neq j$. Then, $\area(C_{\mathrm{reg}})$ is lower bounded by
    \[
F\left(\eta\right)=12\arccot \left( \frac{\cot\tfrac {w-r(T_w)-\eta}{2}\cot\tfrac {r(T_w)+\eta}{2} + \tfrac 12}{\tfrac{\sqrt 3}{2}}\right).
    \]
    which is strictly monotonously increasing in $\eta\in [0,\tfrac w2-r(T_w)]$ and satsifies $F(0) = \area(T_w)$.
\end{lemma}

This lemma yields for a convex body $K$ of width $w\in (0,\tfrac \pi 2]$ that 
\begin{equation}
\label{eq:pal_estimates}
\area(K) \geq \area(C_{\mathrm{reg}}(r(K)-r(T_w))) = F(r(K)-r(T_w)) \geq F(0) = \area(T_w)
\end{equation}
with equality if and only if $r(K)-r(T_w) = 0$ which is equivalent to $K$ being congruent with $T_w$ (cf.\ Lemma \ref{Blaschke}).

The proof of the lemma requires the following area formula for spherical triangles which follows from Napier's rule by a routine computation.

\begin{lemma}\label{TriangleArea}
Let $a,\,b,\,c$ be the side lengths of an arbitrary triangle $T$ in $S^2$, and let $\alpha,\,\beta,\,\gamma$ be the angles opposite these sides. Then, for the area of the triangle, we have
$$
\cot\left(\frac{\area\left(T\right)}{2}\right)=\frac{\cot\frac{a}{2} \cot\frac{b}{2}+\cos\gamma}{\sin\gamma}.
$$
\end{lemma}

\begin{proof}[Proof of Lemma \ref{lemma:hexagon_formular}]
For $\eta\in [0,\tfrac w2-r(T_w)]$, let $B = B(p,r(T_w)+\eta)$. Let $\widetilde{t}_i$ be the intersection point of the line through $\widetilde{q}_i$ and $p$ and the circle $\partial B$, such that $p\in [\widetilde{t_i},\widetilde{q_i}]$. Consider the hexagon
\[
Q\left(\eta\right)=\conv\left(\left\{\widetilde{q}_1,\widetilde{q}_2,\widetilde{q}_3,\widetilde{t}_1,\widetilde{t}_2,\widetilde{t}_3\right\}\right) \subseteq C_{reg}(\eta)
\]
and choose $F(\eta) = \area(Q(\eta))$. From the inclusion, we readily see that $\area(C_{\mathrm{reg}}) \geq F(\eta)$. We can triangulate $Q(\eta)$ with six triangles of the form
\[
T_{ij} = [p,\widetilde{t}_i,\widetilde{q}_j],\quad i\neq j.
\]
For all choices of $i$ and $j$, the triangles $T_{ij}$ are congruent with one another. So we have
\begin{equation}
\label{eq:hex_est}
F(\eta) = \area(Q(\eta)) = 6\area(T_{12}).
\end{equation}
Since $d(p,\widetilde{t}_1) = r(T_w)+\eta$ and $d(p,\widetilde{q}_2) = w-r(T_w)-\eta$, as well as $\angle(\widetilde{t}_1,p,\widetilde{q}_2) = \frac \pi 3$, we obtain from Lemma \ref{TriangleArea}:
\begin{equation}
\label{eq:fs2}
\cot\left(\frac{\area\left(T_{12}\right)}{2}\right)=\frac{\cot\frac{w-r(T_w)-\eta}{2} \cot\frac{r(T_w)+\eta}{2}+\frac{1}{2}}{\frac{\sqrt{3}}{2}}.
\end{equation}
Together with \eqref{eq:hex_est} this verifies the explicit expression for $F(\eta)$ that is claimed in the lemma.

Next, we show that $F(\eta)$ is strictly increasing. Since, clearly, $\area(T_{12}) \in (0,2\pi)$ it suffices to show that 
\[
f\left(\eta\right)=\cot\frac{w-r(T_w)-\eta}{2} \cot\frac{r(T_w)+\eta}{2},
\]
is strictly decreasing, because the cotangens function is strictly decreasing on $(0,\pi)$.
Computing the derivative of $f(\eta)$ yields
\begin{equation}
\label{eq:derif}
f'\left(\eta\right)=\frac{\sin\left(r+\eta\right)-\sin\left(w-r-\eta\right)}{4\sin^2\left(\frac{r+\eta}{2}\right)\sin^2\left(\frac{w-r-\eta}{2}\right)}<0,
\end{equation}
since $0<r+\eta<w-r-\eta<\frac{\pi}{2}$.
In order to see that $F(0)=0$ we observe that the hexagon $Q(\eta)$ degenerates to a regular triangle with inradius $r(T_w)$ if $\eta=0$.
\end{proof}

\subsection{Stability of the spherical isominwidth problem for \boldmath$w\leq\frac{\pi}{2}$}\label{subsec:stab:smallw}

We prove the following variant of Theorem \ref{stab:thm} for small values $\varepsilon$. The actual Theorem \ref{stab:thm} follows trivially from the theorem below, since the Hausdorff distance of two spherical convex bodies is never more than $\pi$. So we can simply replace the constant $\cc_w$ in Theorem \ref{stab:thm}a by $\max\{\cc_w, \pi/\varepsilon_w\}$.

\vspace{6pt}\noindent{\bf Theorem~\ref{stab:thm}a. }\emph{
  Let $T_w\subset S^2$ denote a regular triangle of width $0<w\leq\frac{\pi}{2}$. There exists constants $\cc_w,\varepsilon_w>0$ depending only on $w$ such that for any $\varepsilon\in [0,\varepsilon_w)$ and any convex body $K\subset S^2$ with $w(K) = w$ and
    $\area(K)\leq\area(T_w) + \varepsilon$ we have $\delta(K,T)\leq\cc_w\varepsilon$ for a certain regular triangle $T$ of width $w$.
}

\medskip
In the following, we will use the variables $\cc_w $ and $\varepsilon_w$ as placeholders for positive constants whose value depend only on $w$. The exact value of these constants may vary at each appearance. Nonetheless, we occasionally use variants such as $\cc_w^\prime$ to highlight the different values.

As a first step, we want to control the inradius of convex bodies $K$ whose area is close to the area of the regular triangle of the same width.

\begin{lemma}
    \label{stab:inradius}
    Let $T_w\subset S^2$ denote a regular triangle of width $0<w\leq \tfrac \pi 2$. There exists a constant $\cc_w>0$ depending only on $w$ such that for any convex body $K\subset S^2$ with $w(K)=w$ and $\area(K)\leq \area(T_w)+\varepsilon$ we have $r(K) \leq r(T_w) +\cc_w\varepsilon$.
\end{lemma}

For the proof, we need the following elementary lemma.

\begin{lemma}
    \label{stab:taylor}
    Let $f\colon [a,b]\to \R$ be a continuously differentiable function that attains its unique minimum at $a$. If $f^\prime(a) >0$, there exists a constant $c>0$ such that $|x-a|\leq c |f(x)-f(a)|$, for all $x\in [a,b]$.
\end{lemma}

\begin{proof}
    Since the derivative $f^\prime$ is continuous, there exists a neighbourhood $[a,a+\delta)$ of $a$ on which $f^\prime(x)\geq \half f^\prime(a)>0$. By the mean value theorem, it follows that $|f(x)-f(a)|\geq \half f^\prime(a)|x-a|$ holds for $x\in [a,a+\delta)$.
    
   On the other hand, since $a$ is the unique minimum of $f$ and since $[a+\delta, b]$ is a compact interval, there exists a number $m>f(a)$ such that $f(x)\geq m$, for all $x\in [a+\delta,b]$. Hence,
   \[
   	f(a) + \frac{m-f(a)}{b-a}(x-a)\leq m\leq f(x),
   \]
   for all $x\in [a+\delta,b]$. The claim now follows with $c^{-1}=\min\left\{\half f^\prime(a), \tfrac{m-f(a)}{b-a}\right\}$.
\end{proof}

\begin{proof}[Proof of Lemma \ref{stab:inradius}]
Consider the function $F\colon [0,\tfrac{w}{2}-r(T_w)]\to\R$ as given by Lemma \ref{lemma:hexagon_formular}. $F$ is strictly increasing and continuously differentiable with $F'(0)>0$ (cf.\ \eqref{eq:derif}). Thus, Lemma \ref{stab:taylor} gives
\[
r(K)-r(T_w) \leq \cc_w(F(r(K)-r(T_w)) - F(0)).
\]
Note that $r(K)-r(T_w)$ is non-negative by Lemma \ref{Blaschke}. According to Lemma \ref{lemma:hexagon_formular} we have $F(0) = \area(T_w)$ and by \eqref{eq:pal_estimates} we have $F(r(K)-r(T_w))\leq \area(K)$. We deduce
\[
    r(K)-r(T_w) \leq \cc_w(\area(K) - \area(T_w) )\leq \cc_w\varepsilon,
\]
where we used the assumption on $K$ for the last inequality.
\end{proof}

The challenge in the proof of Theorem \ref{stab:thm} is to find a suitable regular triangle $T$, which is close to $K$. First, we consider the cap domain $C\subseteq K$ from \eqref{eq:capdomaindef} as it has been guaranteed in Proposition \ref{prop:capdomain}. Next, we ``regularize'' $C$ by considering the regularized capdomain $\creg$ as defined in \eqref{eq:regcapdomaindef}. Finally, our candidate for $T$ is the regular triangle of width $w$ centered at the incenter $p$ of $K$ whose vertices $v_i$ satisfy $\widetilde{q}_i\in [p,v_i]$. The proof of Theorem \ref{stab:thm} is done in three steps, out of which the first one is going to take the most effort.
\begin{enumerate}\item[]\begin{enumerate}
    \item[Step 1:] $\delta(K,C)<\cc_w\varepsilon$, for $\varepsilon<\varepsilon_w$.
    \item[Step 2:] $\delta(C,\creg) < \cc_w\varepsilon$, for $\varepsilon<\varepsilon_w$.
    \item[Step 3:] $\delta(\creg,T) < \cc_w\varepsilon$.
\end{enumerate}\end{enumerate}%
Before we come to the proofs of the three individual inequalities, we need another lemma concerning the angles between the contact points of $K$ with its incircle. In the setting of Theorem \ref{stab:thm} we may assume that there are no two opposite contact points (in which case $K$ is a ball with radius $\tfrac w 2$), but there exists three distinct contact points $t_1,t_2,t_3$ such that $p\in\inte [t_1,t_2,t_3]$.

\begin{lemma}
    \label{stab:contacts}

    Let $\eta = r(K) - r(T_w)$. The following estimates hold:
    \begin{enumerate}
        \item  Let $C_i\subset C$ be the cap with peak $q_i$. It is attached to the incircle $B$ of $K$ at two points $s_{i,j}$ and $s_{i,k}$, $\{i,j,k\}=\{1,2,3\}$. We have $\angle(s_{i,j},p,s_{i,k}) \geq \frac{2\pi}{3} - \cc_w\eta$.
        \item  For the angles between two of the contact points $t_i$, we have
        \[
    \frac{2\pi}{3}-\cc_w\eta \leq \angle(t_i,p,t_j) \leq \frac{2\pi}{3} + \cc_w\eta.
    \]
        \item For the angles between two of the peaks of $C$, we have
        \[
        \frac{2\pi}{3}-\cc_w\eta \leq \angle(q_i,p,q_j) \leq \frac{2\pi}{3} + \cc_w\eta.
        \]
    \end{enumerate}
\end{lemma}

In particular, since $\eta<\cc_w\varepsilon$ by Lemma \ref{stab:inradius},
\begin{equation}
    \frac{2\pi}{3}-\cc_w\varepsilon \leq \angle(t_i,p,t_j) \leq \frac{2\pi}{3} + \cc_w\varepsilon.
\end{equation}

\begin{proof}
    For (1), we may assume $i=1$ and we let $\alpha=\angle(s_{1,2},p,s_{1,3})$. We consider the triangle $T=[p,s_{1,2},q_1]$. The following data on $T$ is known:
    \begin{itemize}
        \item $d(p,s_{1,2}) = r(K)$,
        \item $d(p,q_1) = w-r(K)$,
        \item $\angle(q_1,p,s_{1,2}) = \tfrac \alpha 2$,
        \item $\angle(p,s_{1,2},q_1) = \tfrac \pi 2$.
    \end{itemize}
    Using spherical trigonometry, we can express $\sin(\alpha/2)$ via a function $g\colon [r(T_w),w/2]\to \R$ of the inradius of $K$:
    Let $d = d(s_{1,2},q_1)$. By the spherical Pythagorean theorem, we have $\cos(w-r) = \cos(r)\cos(d)$. Hence, $\sin(d) = 1-\tfrac{\cos(w-r)^2}{\cos(r)^2}$. The spherical law of sines gives
    \[
    g(r) = \frac{\sqrt{1-\frac{\cos(w-r)^2}{\cos(r)^2}}}{\sin(w-r)}.
    \]

    \noindent Since $0<r(T_w) < r < w/2<w$, the function $g$ is continuously differentiable on its compact domain $[r(T_w),w/2]$ and, in particular, Lipschitz continuous with a Lipschitz constant depending only on the width $w$.
    It follows that
    $$
    |g(r(T_w)) - g(r(K))| \leq \cc_w\eta.
    $$
    Note that $g(r(T_w)) = \sin(\tfrac \pi 3)$, since for $r(K)=r(T_w)$ we must have that $K$ is a regular triangle (cf.\ Lemma \ref{Blaschke}), which means that the points $s_{1,2}$ and $s_{1,3}$ coincide with the the contact points $t_2$ and $t_3$. Hence, we find
    \begin{equation}
    \label{eq:sin_est}
    |\sin(\tfrac \alpha 2) - \sin(\tfrac \pi 3) | \leq \cc_w\eta.
    \end{equation}

    Since $w\leq \tfrac \pi 2$, we have
    \[
        g(r) \leq \frac{\sqrt{1-\cos(w-r)^2}}{\sin(w-r)^2} = \sin(w-r)\leq \sin(w-r(T_w)) <1,
    \]
    the claim thus follows from \eqref{eq:sin_est} and the Lipschitz continuity of $\arcsin$ on the interval $[-\sin(w-r(T_w), \sin(w-r(T_w)]\subset (-1,1)$. 

    For (2), we observe that $\angle(t_i,p,t_j) \geq \angle(s_{k,i},p,s_{k,j})$, where $k$ is such that $\{i,j,k\}=\{1,2,3\}$. Hence, the lower bound follows from (1). But since the angles between the contact points sum up to $2\pi$, the upper bound follows as well.

    For (3), we notice that \[
    \angle(q_i,p,q_j)\geq \angle(q_i,p,s_{i,j}) + \angle(s_{j,i},p,q_j) = \half\angle(s_{i,k},p,s_{i,j}) + \half\angle(s_{j,i},p,s_{j,k}). 
    \]
    So, just as in (2), the lower bound follows from (1) and the upper bound follows from the lower bound and the fact that the angles in question sum up to $2\pi$.
\end{proof}

\noindent\textbf{Step 1: \boldmath$\delta(K,C)<\cc_w\varepsilon$, for $\varepsilon<\varepsilon_w$.}
Let $t_1,t_2,t_3\in B\cap K$, where $B$ denotes the incircle of $K$, be the contact points of $\partial B$ and $\partial K$ such that the incenter $p$ of $B$ is contained in $\inte [t_1,t_2,t_3]$. Then the three tangent lines $\ell_i$, $1\leq i\leq 3$ to $B$ at $t_i$ define a triangle $T$ such that $C\subseteq K\subseteq T$. Let $p_i$ be the vertex opposite to $\ell_i$ of this triangle. Our goal is to show that the peak $q_i$ of the corresponding cap in $C$ has distance at most $\cc_w\varepsilon$ to $p_i$. Since $\max_{x\in T}d(x,C)$ is attained at a vertex of $T$ and $K\subseteq T$, this shows $\max_{x\in K}d(x,C)<\cc_w\varepsilon$. Since $C\subseteq K$, this would finish the proof of Step 1.

In order to bound the distance, we recall that $d(q_i,p) = w-r(K)$ by construction. We aim to show first that
\begin{equation}
    \label{stab:eq:dpp}
    d(p_i,p)\leq d(q_i,p) + \cc_w \varepsilon.
\end{equation} 
For the sake of simplicity, let $i=1$ in the following. It is clear that for a fixed radius $r(K)$, $d(p_1,p)$ is strictly increasing, as the angle $\angle(t_2,p,t_3)$ increases. We saw in Lemma \ref{stab:contacts} that this angle is bounded from above by $\tfrac{2\pi}{3}+\cc_w\eta$, where $\eta = r(K) - r(T_w)$. Thus, the distance $d(p_i,p)$ is not longer than the hypotenuse of the right triangle $T^\prime$ given by the incenter $p$, a point $t\in\partial B$ with $\angle(p_i,p,t) = \tfrac\pi 3 + \cc_w\eta$ and a third vertex $p_1^\prime$ at the intersection of the tangent line of $B$ at $t$ and the line generated by $p$ and $p_1$ (cf.\ Figure \ref{fig:Tprime}).

\begin{center}
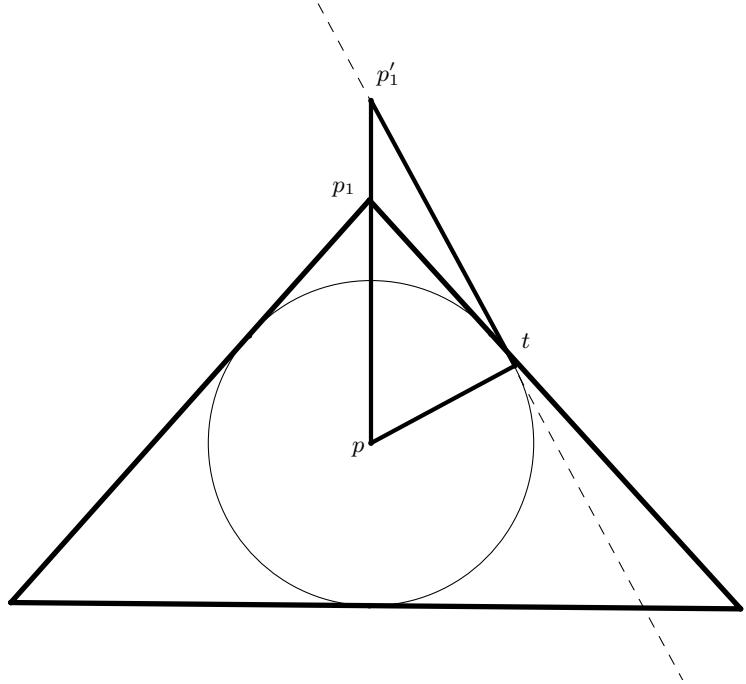
\begin{figure}[H]
\begin{tikzpicture}[line cap=round,line join=round,>=triangle 45,x=3.0cm,y=3.0cm, scale = .7]
\clip(-3,-1.5) rectangle (3,2.8);
\draw(-2.92,3.83) -- (-1.66,3.83);
\draw(0,0) circle (3.09cm);
\draw [dash pattern=on 4pt off 4pt,domain=-3:3] plot(\x,{(--1.06-0.91*\x)/0.49});
\draw[line width=1.6pt] (0,2.17) -- (0.91,0.49) -- (0,0) -- (0,2.17);
\draw[line width=2pt] (-2.28,-1.01) -- (-0.01,1.54) -- (2.34,-1.05) -- (-2.28,-1.01);
\begin{scriptsize}
\fill [color=black] (-2.74,3.83) circle (1.5pt);
\draw[color=black] (-2.46,4.06) node {$eta = 0.03$};
\fill [color=black] (0,2.17) circle (1.5pt);
\draw[color=black] (0.11,2.34) node {$p_1^\prime$};
\fill [color=black] (0.91,0.49) circle (1.5pt);
\draw[color=black] (0.98,0.65) node {$t$};
\fill [color=black] (0,0) circle (1.5pt);
\draw[color=black] (-0.08,-0.04) node {$p$};
\fill [color=black] (-0.77,0.68) circle (1.5pt);
\fill [color=black] (0.76,0.69) circle (1.5pt);
\fill [color=black] (-0.01,-1.03) circle (1.5pt);
\fill [color=black] (-2.28,-1.01) circle (1.5pt);
\fill [color=black] (-0.01,1.54) circle (1.5pt);
\draw[color=black] (-0.17,1.61) node {$p_1$};
\fill [color=black] (2.34,-1.05) circle (1.5pt);
\end{scriptsize}
\end{tikzpicture}
\caption{The triangle $T^\prime = [p,t,p_1^\prime]$ used to estimate the distance $d(p,p_1)$}\label{fig:Tprime}
\end{figure}
\end{center}

In this triangle, all the side lengths and angles that are given depend solely on $r(K)$ and $w$. Thus, we can upper bound $d(p_i,p)$, by $g(r(K)) =d(p_i^\prime,p)$, a function whose only variable is the inradius of $K$. We note that for $r(K)=r(T_w)$, we have $q_i=p_i=p_i^\prime$ and, thus, $g(r(T_w)) = w-r(T_w)$.

In order to show \eqref{stab:eq:dpp}, we show that the function $g(r)$ is Lipschitz continuous in a neighbourhood of $r(T_w)$. Note that $g(r)$ is well-defined for $0<r<r(T_w)$ as well.

Let $l$ be the length of the edge of $T^\prime$ opposite to $p$. Then, by the spherical Pythagorean theorem, $\cos(g(r)) = \cos(r)\cos(l)$. Since the angle at $p$ is given by $\tfrac{\pi}{3}+\cc_w\eta$, we can compute $l$ using the spherical law of sines and find 
\begin{equation}
\label{stab:eq:l}
    \sin\left(l\right) = \sin\left(\tfrac{\pi}{3}+\cc_w(r-r(T_w))\right)\sin(g(r)).
\end{equation}
Substituting this in the formula obtained by the Pythagorean theorem yields the following implicit description of $g(r)$:
\begin{equation}
\label{stab:eq:implicit}
    \cos(g(r)) = \cos(r)\sqrt{1-\big(\sin(\tfrac{\pi}{3}+\cc_w(r-r(T_w)))\sin(g(r))\big)^2}.
\end{equation}
We seek to apply the implicit function theorem (see Walter \cite[Theorem 12.2.A]{Wal}) in order to see that $g(r)$ is Lipschitz continuous. To this end, let 
\begin{equation}
    F(g,r) = \cos(g) -  \cos(r)\sqrt{1-\sin(\tfrac{\pi}{3}+\cc_w(r-r(T_w)))^2\sin(g)^2}
\end{equation}
be defined in a suitable neighbourhood of $(g,r) = (w-r(T_w),r(T_w))$. Since $g(r(T_w)) = w-r(T_w)$, it follows from \eqref{stab:eq:implicit} that $F(w-r(T_w),r(T_w)) = 0$. Moreover, we have

\begin{gather*}
\frac{\partial}{\partial g} F(w-r(T_w),r(T_w)) = \sin(w-r(T_w)) \left( \frac{\tfrac{3}{4}\cos(r(T_w))\cos(w-r(T_w))}{\sqrt{1-\sin(\tfrac{\pi}{3})^2\sin(w-r(T_w))^2}}-1 \right)= \\
 = \sin(w-r(T_w))\big(\tfrac{3}{4}\cos(r(T_w))^2 -1\big)<0,
\end{gather*}

where we used \eqref{stab:eq:implicit} and the Pythagorean theorem on the sphere to obtain the second line. By the implicit function theorem, there exists a neighbourhood $U$ of $r(T_w)$ and a continuously differentiable function $\widetilde{g}:U\to\R$ such that $F(\widetilde{g}(r),r)=0$, for all $r\in U$. Since this function $\widetilde{g}$ is unique, it agrees with $g$ on $U$.
Since
\[
 \frac{\partial}{\partial r} F(w-r(T_w),r(T_w)) = \frac{\sqrt{3}\cc_w\cos(r(T_w))\sin(w-r(T_w))^2}{4\sqrt{1-\tfrac 34 \sin(w-r(T_w))^2}} + \sqrt{1-\tfrac 34 \sin(w-r(T_w))^2}\sin(r(T_w))
\]
we have by the implicit function theorem
\[
g'(r(T_w)) = - \left(\frac{\partial}{\partial g} F(w-r(T_w),r(T_w))\right)^{-1} \frac{\partial}{\partial r} F(w-r(T_w),r(T_w)) >0.
\]
So there exists a compact neighbourhood $V\subset U$ of $r(T_w)$ on which the derivative of $g$ is bounded from below. In particular $g$ is Lipschitz continuous on $V$ and $V$ depeneds only on $w$.
If $\varepsilon<\varepsilon_w$, for a suitable number $\varepsilon_w$, then $r(K)$ falls into this neighbourhood (cf.\ Lemma \ref{stab:inradius}). Hence, for $\varepsilon<\varepsilon_w$, we have with $\eta = r(K)-r(T_w) <\cc_w\varepsilon$ the following estimate:
\[\begin{split}
d(p_1,p) &\leq g(r(K)) \leq \cc_w\eta+g(r(T_w)) = \cc_w\eta + (w-r(T_w))\\
&\leq \cc_w^\prime\eta + w-r(K) = \cc_w^\prime\eta + d(q_1,p),
\end{split}\]
which proves \eqref{stab:eq:dpp} (cf.\ Lemma \ref{stab:inradius}).

In order to bound $d(q_1,p_1)$, we consider the points $x_2$ and $x_3$ at which the tangent lines $\ell_2$ and $\ell_3$ intersect $B(p,w-r(K))$. By construction, $q_i$ is located on the shorter arc $\arc{x_2x_3}\subset\partial B(p,w-r(K))$. Since disks of radius less than $\tfrac \pi 2$ are convex, $d(p_i,q_i)$ is upper bounded by the distance of $p_i$ to the segment $[x_2,x_3]$. This distance is maximized by the two endpoints $x_i$ ($x_2$, say). Since $t_2$, $x_2$ and $p_1$ lie subsequently on the line $\ell_2$, we have (cf.\ Figure \ref{fig:ell2})
\[
d(p_1,q_1)\leq d(x_2,p_1) = d(t_2,p_1) - d(t_2,x_2).
\]
Using the Pythagorean theorem for the triangles $[t_2,p,x_2]$ and $[t_2,p,p_1]$, the two distances on the right hand side can be computed (recall that $d(x_2,p) = d(q_1,p)$). 
\begin{equation}
\label{eq:diff_est}
d\left(t_2,p_1\right)-d\left(t_2,x_2\right)=
\arccos\left(\frac{\cos(d(p_1,p))}{\cos(r(K))}\right) - \arccos\left(\frac{\cos(d(q_1,p))}{\cos(r(K))}\right).
\end{equation}

\begin{center}
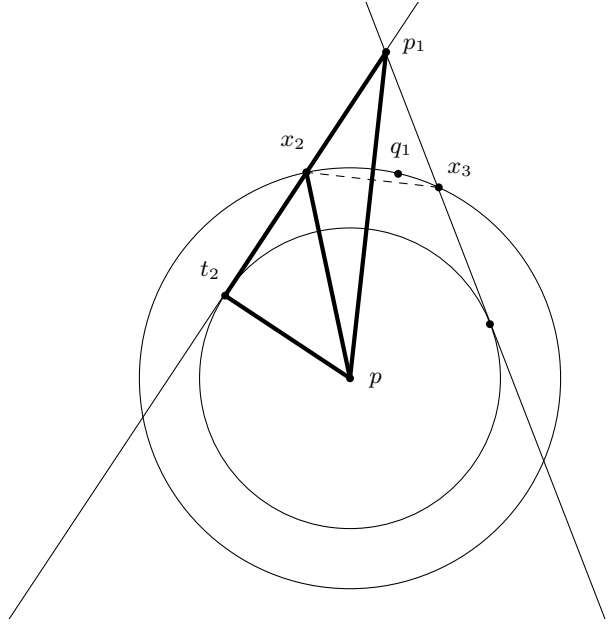
\begin{figure}[H]
\begin{tikzpicture}[line cap=round,line join=round,>=triangle 45,x=2.0cm,y=2.0cm]
\clip(-2.3,-1.6) rectangle (2.3,2.5);
\draw(0,0) circle (2cm);
\draw [line width=0.4pt] (0,0) circle (2.8cm);
\draw [domain=-2.3:2.3] plot(\x,{(--1--0.83*\x)/0.55});
\draw [domain=-2.3:2.3] plot(\x,{(--1-0.93*\x)/0.36});
\draw [dash pattern=on 3pt off 3pt] (0.59,1.27)-- (-0.29,1.37);
\draw [line width=1.6pt] (0,0)-- (-0.83,0.55);
\draw (-0.83,0.55)-- (-0.29,1.37);
\draw [line width=1.6pt] (-0.29,1.37)-- (0,0);
\draw [line width=1.6pt] (0.24,2.17)-- (0,0);
\draw [line width=1.6pt] (-0.83,0.55)-- (0.24,2.17);
\begin{scriptsize}
\fill [color=black] (-0.83,0.55) circle (1.5pt);
\draw[color=black] (-0.93,0.72) node {$t_2$};
\fill [color=black] (0.93,0.36) circle (1.5pt);
\fill [color=black] (-0.29,1.37) circle (1.5pt);
\draw[color=black] (-0.38,1.57) node {$x_2$};
\fill [color=black] (0.59,1.27) circle (1.5pt);
\draw[color=black] (0.73,1.39) node {$x_3$};
\fill [color=black] (0.24,2.17) circle (1.5pt);
\draw[color=black] (0.43,2.21) node {$p_1$};
\fill [color=black] (0,0) circle (1.5pt);
\draw[color=black] (0.17,-0.01) node {$p$};
\fill [color=black] (0.32,1.36) circle (1.5pt);
\draw[color=black] (0.34,1.52) node {$q_1$};
\end{scriptsize}
\end{tikzpicture}
\caption{The two triangles $[t_2,p,x_2]$ and $[t_2,p,p_1]$ used to bound $d(p_1,q_1)$}%
\label{fig:ell2}
\end{figure}
\end{center}

Since $r(K)\leq r(T_w)+\cc_w\varepsilon$ and   $d(p_1,p)\geq d(q_1,p)=w-r(K)\geq w-r(T_w)-\cc_w\varepsilon$ it follows that the arguments of the $\arccos$ function in \eqref{eq:diff_est} are bounded from above by $\cos(w-r(T_w)-\cc_w\varepsilon)/\cos(r(T_w)+\cc_w\varepsilon)$. For $\varepsilon_w$ small enough and $\varepsilon<\varepsilon_w$ this number is less than 1. So we can use the Lipschitz continuity of $\arccos$ on compact subintervals of $(-1,1)$ to deduce

\[\begin{split}
    d(p_1,q_1) \leq d(t_2,p_1) - d(t_2,x_2) \leq\cc_w \left| \frac{\cos(d(p_1,p)) -\cos(d(q_1,p))}{\cos(r(K))}\right|\leq \cc_w' | d(p_1,p) -d(q_1,p)| \leq \cc_w'' \varepsilon,
\end{split}\]

where we used $r(K)\leq \tfrac w2$, the Lipschitz continuity of $\cos$ and \eqref{stab:eq:dpp}. The first step is now proven.

\noindent\textbf{Step 2: \boldmath$\delta(C,\creg) < \cc_w\varepsilon$ for $\varepsilon<\varepsilon_w$.}
Since the extreme points of the two cap domains are given by a subset of the boundary points of the incircle of $K$ together with the three peaks $q_i$, respectively $\widetilde{q}_i$, it suffices to show 
\begin{equation}
    \label{stab:eq:peaks}
    d(q_i,\widetilde{q}_i) < \cc_w\varepsilon. 
\end{equation}
For $i=1$, we have by definition $q_1=\widetilde{q}_1$, so let $i\in\{2,3\}$. By construction, we have
\begin{equation}
\label{eq:angle_est}
\angle(q_i,p,\widetilde{q}_i) = \left|\angle(q_i,p,q_1) - \angle(\widetilde{q_i},p,q_1)\right|=\left|\tfrac{2\pi}{3} - \angle(q_i,p,q_1)\right|\leq \cc_w\varepsilon,
\end{equation}
where we used item (3) of Lemma \ref{stab:contacts} and Lemma \ref{stab:inradius} to obtain the last inequality. In particular, for $\varepsilon<\varepsilon_w$ and $\varepsilon_w$ small enough, we have $\angle(q_i,p,\widetilde{q}_i)< \tfrac \pi 2$.
Hence, using the spherical law of cosines,
\[\begin{split}
\cos(d(q_i,\widetilde{q}_i)) &= \cos(w-r(K))^2 + \sin(w-r(K))^2\cos(\angle(q_i,p,\widetilde{q}_i))\\
&= \cos(w-r(K))^2 + (1-\cos(w-r(K))^2)\cos(\angle(q_i,p,\widetilde{q}_i))\\
&= \cos(\angle(q_i,p,\widetilde{q}_i)) + \cos(w-r(K))^2(1-\cos(\angle(q_i,p,\widetilde{q}_i)))\\
&\geq \cos(\angle(q_i,p,\widetilde{q}_i)).
\end{split}\]
Thus, \eqref{stab:eq:peaks} follows from the monotonicity property of $\arccos$ together with \eqref{eq:angle_est}.

\noindent\textbf{Step 3: \boldmath$\delta(\creg,T) < \cc_w\varepsilon$.} Let $v_i$ denote the vertices of $T$ and let $\widetilde{q}_i$ denote the peaks of $\creg$ for $1\leq i\leq 3$. By construction, we have $d(v_i,\widetilde{q}_i) = r(K)-r(T_w)$, which is less than $\cc_w\varepsilon$ by Lemma \ref{stab:inradius}. This shows already that $\max_{x\in T}d(x,\creg) < \cc_w\varepsilon$. On the other hand, those extreme points of $\creg$ that are not among the three peaks of its caps, lie on the boundary of the incircle $B(p,r(K))$ of $K$. By construction, $T$ is centered at $p$ as well and, thus, contains the disk $B(p,r(T_w))$. Hence, any extreme point $x$ of $\creg$ satisfies $d(x,T)\leq r(K)-r(T_w)\leq \cc_w\varepsilon$, which concludes the proof of the final step.\hfill $\Box$\\

The next proposition shows that the dependence on $\varepsilon$ in the bound $\delta(K,T_w)\leq \cc_w\varepsilon$ cannot be improved. The construction is illustrated in Figure \ref{fig:opt}.

\begin{prop}
 \label{prop:optimality}
 Let $w>0$ and consider a regular triangle $T_w = [v_1,v_2,v_3]\subset S^2$ of width $w \leq \tfrac \pi 2$. For $\varepsilon>0$ let  $K_\varepsilon = [v_1,v_2,v_3^\varepsilon]$, where $v_3^\varepsilon$ is the point on the geodesic line through $v_2$ and $v_3$ such that $d(v_3,v_3^\varepsilon)=\varepsilon$ and $v_3\in [v_2,v_3^\varepsilon]$ with the assumption that $v_1$, $v_2$ and $v_3^\varepsilon$ are contained in an open hemisphere. The following statements hold:
 \begin{enumerate}
     \item $w(K_\varepsilon) = w$ for any $\varepsilon<\varepsilon_w$, where $\varepsilon_w>0$ is a constant depending only on $w$,
     \item $\delta(K_\varepsilon,T) \geq \tfrac \varepsilon 2$ for any regular triangle $T$ of width $w$,
     \item $\area(K_\varepsilon) \leq \area(T_w)+\cc_w\varepsilon$, for all $\varepsilon<\varepsilon_w$, where $\cc_w$ and $\varepsilon_w$ are constants that depend only on $w$.
 \end{enumerate}
\end{prop}

\begin{proof}
    We start with item (1). Let $t$ be the midpoint of $[v_2,v_3]$. Since $w\leq \tfrac \pi 2$, the width of $T_w$ is realized by the lune $L$ given by the supporting line defined by $v_2$ and $v_3$ and the supporting line of $T_w$ at $v_1$ which is orthogonal to $[t,v_1]$. For $\varepsilon>0$ small enough we have $K_\varepsilon\subset L$. Hence, $w(K_\varepsilon)\leq w$. Since $T_w\subseteq K_\varepsilon$, the statement (1) follows.

    For (2) let $T$ be any regular triangle. Considering the segment $[v_2,v_3^\varepsilon]\subset K_\varepsilon$ shows $\diam(K_\varepsilon) \geq \diam(T) + \varepsilon$. Let $\delta =\delta(K_\varepsilon,T)$, then for each $x\in K_\varepsilon$ there exists $y\in T$ with $d(x,y) \leq \delta$. Let $[x_1,x_2]\subset K_\varepsilon$ be a segment of maximal length and let $y_1,y_2\in T$ be such that $d(x_i,y_i)\leq \delta$. Then,
    \[
    \diam(T)+\varepsilon \leq \diam(K_\varepsilon) = d(x_1,x_2) \leq 2\delta + d(y_1,y_2) \leq 2\delta + \diam(T).
    \]

    For (3), let $T_\varepsilon = [v_1,v_3,v_3^\varepsilon]$, then it suffices to show $\area(T_\varepsilon) \leq \cc_w\varepsilon$. We denote by $a_w$ and $\alpha_w$ the side length and angle of $T_w$ respectively. Using Lemma \ref{TriangleArea} we obtain
    \[
    \area(T_\varepsilon) = 2\arccot\left( \frac{\cot\frac{a_w}{2}\cot\frac{\cc_w\varepsilon}{2}+\cos\alpha_w}{\sin\alpha_w} \right)=:f_w(\varepsilon).
    \]
    It can be checked that $f_w$ is differentiable at $\varepsilon=0$ with $f_w'(0) = \sin(\alpha_w)\tan(a_w/2) =\colon \cc_w / 2>0$. Thus, there exists $\varepsilon_w>0$ such that for all $\varepsilon\in [0,\varepsilon_w)$ we have
    \[
        \area(T_\varepsilon) = f_w(\varepsilon) \leq \cc_w\varepsilon + f_w(0) = \cc_w\varepsilon.\qedhere
    \]
\end{proof}

\begin{figure}[H]
\begin{center}
\begin{tikzpicture}[line cap=round,line join=round,>=triangle 45,x=1cm,y=1cm,scale=9]
\clip(-0.05,-0.05) rectangle (0.8,0.6);
\draw [line width=0.8pt] (0.3033390205427441,0.525398595498212)-- (0.7370943494944908,0);
\draw [line width=0.8pt] (0.6066780410854878,0)-- (0.7370943494944908,0);
\draw [line width=0.8pt] (0,0)-- (0.6066780410854878,0);
\draw [line width=0.5pt,dashed] (0.6066780410854878,0)-- (0.3033390205427441,0.525398595498212);
\draw [line width=0.8pt] (0.3033390205427441,0.525398595498212)-- (0,0);
\draw (0.31,0.55) node[anchor=center] {$v_1$};
\draw (-0.01,-0.058) node[anchor=south] {$v_2$};
\draw (0.6070271014891571,-0.058) node[anchor=south] {$v_3$};
\draw (0.75,-0.058) node[anchor=south] {$v_3^\varepsilon$};
\draw (0.67188619529,0.027) node[anchor=center] {$\varepsilon$};
\draw [line width=0.5pt,dotted] (0.3033390205427441,0.525398595498212)-- (0.3033390205427441,0);
\draw (0.325,0.26969066668386005) node[anchor=center] {$w$};
\begin{scriptsize}
\draw [fill=black] (0,0) circle (.2pt);
\draw [fill=black] (0.6066780410854878,0) circle (.2pt);
\draw [fill=black] (0.3033390205427441,0.525398595498212) circle (.2pt);
\draw [fill=black] (0.7370943494944908,0) circle (.2pt);
\end{scriptsize}
\end{tikzpicture}
\caption{An ``almost regular'' triangle showing the optimality of the stability result}
\label{fig:opt}
\end{center}
\end{figure}
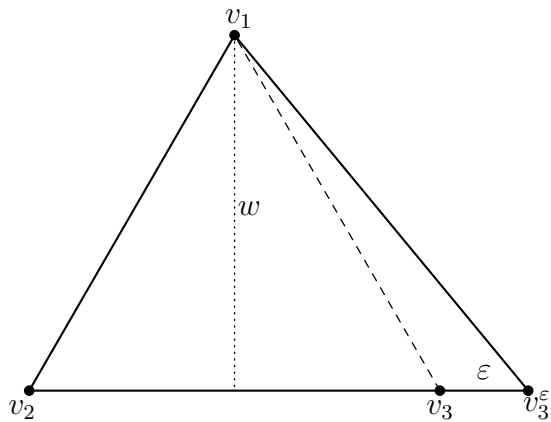

\section{P\'al's problem for \boldmath$w>\frac{\pi}{2}$}
\label{sec:pal:largew}

\subsection{The spherical isominwidth inequality for $w>\frac{\pi}{2}$.}
\label{subsec:pal:largew}
Here we prove Theorem \ref{intro:thm:largewidth}.
Let us recall that a Reuleaux triangle $U_r$ in $S^2$ is an intersection of three balls of radius $r$, placed at the vertices of a regular triangle with edge length $r$. It is well-known that $U_r$ is a body of constant width $r\leq \tfrac{\pi}{2}$.

The following proposition provides an explicit description of the polar of a Reuleaux triangle.

\begin{prop}
    \label{prop:uwcirc}
    Let $U_w\subset S^2$ be a Reuleaux triangle given as the intersection of three circles $B(x_i,w)\subset S^2$, $1\leq i\leq 3$, where $d(x_i,x_j)=w\leq\tfrac \pi 2$, for $i\neq j$. Then,
    \begin{equation}
    \label{eq:polarreuleaux}
    U_w^\circ = -\conv\left\{ B(x_1,\tfrac{\pi}{2}-w)\cup B(x_2,\tfrac{\pi}{2}-w)\cup B(x_3,\tfrac{\pi}{2}-w)   \right\}.
    \end{equation}
\end{prop}

\begin{proof}
    By definition, one has $U_w = B(x_1,w)\cap B(x_2,w)\cap B(x_3,w)$. Let $C_i = \pos B(x_i,w)$. Then, $\pos U_w = C_1\cap C_2\cap C_3$ and, thus (cf.\ Schneider \cite[Theorem 1.6.3]{Sch14})
    \[
    \pos(U_w)^\star = \conv_{\R^3}(C_1^\star\cup C_2^\star \cup C_3^\star).
    \]
    It is easy to check that $B(x_i,w)^\circ = -B(x_i,\tfrac{\pi}{2}-w)$. Hence,
    \[
    \begin{split}
        \pos(U_w^\circ) = -\conv_{\R^3}\left(\bigcup_{i=1}^3 \pos B(x_i,\tfrac{\pi}{2}-w)\right) = \pos\left( -\conv_{S^2} \left(\bigcup_{i=1}^3 B(x_i,\tfrac{\pi}{2}-w)\right)\right)
    \end{split}
    \]
    and the claim follows.
\end{proof}

By the proposition, Theorem \ref{intro:thm:largewidth} is equivalent to the following statement.

\begin{theorem}
\label{thm:largewidth}
    Let $K\subset S^2$ be a convex body of width $w > \tfrac{\pi}{2}$. Then, 
    \[
    \area(K)\geq\area(U_{\pi-w}^\circ),
    \]
    for a Reuleaux triangle $U_{\pi-w}$ of width $\pi-w$. Equality is obtained, if and only if $K$ is a congruent copy of $U_{\pi-w}^\circ$.
\end{theorem}

In order to prove this theorem, we require two spherical analogs of classical theorems on bodies of constant width in $\R^2$.
First, the Blaschke--Lebesgue Theorem states that the area of a body $K\subset\R^2$ of constant width $w$ is lower-bounded by the area of a Reuleaux triangle of the same width. This result has been proven on the sphere by Leichtweiss \cite{Lei05} and recently also by Bezdek \cite{Bez21} (see also Blaschke \cite{Bla15}, Lebesgue \cite{Leb14}, Ara\'ujo \cite{Ara97} and B\"or\"oczky, Sagmeister \cite{BoS22}).

\begin{theorem}[Leichtweiss]
\label{thm:leichtweiss}
    Let $K\subset S^2$ be a convex body of constant width $w \leq \tfrac{\pi}{2}$. Then,
    \[
    \area(K)\geq\area(U_w),
    \]
    where $U_w\subset S^2$ is a Reuleaux triangle of width $w$. Equality holds, if and only if $K$ is a congruent copy of $U_w$.
\end{theorem}

Next, we need a spherical version of Barbier's theorem on bodies of constant width; it states that a body $K$ of constant width $w$ in $\R^2$ has perimeter $2\pi w$. In particular, the perimeter depends only on the width. On the sphere, this is no longer the case. In order to determine the perimeter of a spherical body $K\subset S^2$ of constant width, one has to consider other geometric properties of $K$. Blaschke gave a spherical version of Barbier's theorem which involves the total curvature of $\partial K$. Here we use a formulation due to Ara\'ujo \cite{Ara96} involving the area of $K$.

\begin{theorem}[Ara\'ujo]
\label{thm:araujo}
Let $K\subset S^2$ be a convex body of constant width $w>0$. Then we have
\begin{equation}
    \per(K) =(2\pi-\area(K))\tan\left(\frac{w}{2}\right),
\end{equation}
where $\per(K)$ denotes the perimeter of $K$.
\end{theorem}

\begin{proof}[Proof of Theorem \ref{thm:largewidth}]
By considering a reduction of $K$, we may assume that $K$ is a reduced convex body. In this case, by Lemma \ref{prop:duality}, $K^\circ$ is of constant width $\pi-w$. In order to give a lower bound on $\area(K)$, we recall the identity
\begin{equation}
\label{eq:areaperimeter}
\area(K)= 2\pi-\per(K^\circ),
\end{equation}
Since $K^\circ$ is of constant width $\pi-w\leq \tfrac{\pi}{2}$, we can
apply Theorem \ref{thm:araujo} to $K^\circ$
and obtain
\begin{equation}\label{eq:area:polarity}
\area(K) = 2\pi-(2\pi-\area(K^\circ))\tan\left(\frac{\pi-w}{2}\right).
\end{equation}
We see that $\area(K)$ is minimal (among reduced bodies of width $w$), if and only if $\area(K^\circ)$ is minimal (among bodies of constant width $\pi -w$). According to Theorem \ref{thm:leichtweiss}, the latter is the case, if and only if $K^\circ$ is a congruent copy of $U_{\pi-w}$.
Since isometries commute with the polarity operation, this is equivalent to $K$ being a congruent copy of $U_{\pi-w}^\circ$.
\end{proof}

In view of Proposition \ref{prop:constwidthpolar}, the body $U_{\pi-w}^\circ$ is of constant width. Thus, Theorem \ref{thm:largewidth} may also be regarded as an extension of Theorem \ref{thm:leichtweiss} to the case $w>\tfrac \pi 2$. 
In the limit case $w\searrow \tfrac{\pi}{2}$ we find that $U_{\pi/2}$ is in fact a regular triangle of width $\tfrac \pi 2$, since a circle of radius $\tfrac \pi 2$ on $S^2$ is a great circle, i.e., a spherical line. It is not hard to see that the polar body of $U_{\pi/2}$ is itself a regular triangle of width $\tfrac \pi 2$. Thus, for $w\searrow \tfrac{\pi}{2}$, Theorem \ref{thm:largewidth} is accordant with Theorem \ref{isominwidth}.

To conclude, we consider the regular triangle $T=[x_1,x_2,x_3]\subset U_w$ of diameter $w$ given by the three intersection points of the circle $B(x_i,w)$ in $U_w$. Then $T$ has the same diameter and circumradius as $U_w$. Hence, $T^\circ \supseteq U_w^\circ$ has the same width and inradius as $U_w^\circ$. Indeed, $U_w^\circ$ is a reduction of a regular triangle of width $\pi-w$. This implies, for instance, that if $w>\tfrac\pi 2$ we have equality in Lemma \ref{Blaschke} not only for the regular triangle $T_w$ but for any convex body $K$ with $U_{\pi-w}^\circ\subseteq K \subseteq T_w$.

\subsection{Stability of the spherical isominwidth problem for \boldmath$w>\frac{\pi}{2}$}\label{subsec:stab:largew}

We saw in the previous section that there is a duality between the Blaschke-Lebesgue inequality for spherical bodies of constant width at most $\tfrac \pi 2$ and the isominwidth inequality for reduced bodies of width at least $\tfrac \pi 2$. Therefore, the following result from B\"or\"oczky, Sagmeister \cite{BoS22} will be helpful.

\begin{theorem}\label{BL-stab}
Let $0<D<\frac{\pi}2$.
If $K\subset S^2$ is a body of constant width $D$, $\varepsilon\geq 0$ and  
$$
\area\left(K\right)\leq (1+\varepsilon)\area\left(U_D\right),
$$
then there exists a Reuleaux triangle $U\subset S^2$ of width $D$ such that
$\delta(K,U)\leq \theta\varepsilon$ where $\theta>0$ is an explicitly calculable constant depending on $D$.
\end{theorem}

Since the duality between the Blaschke-Lebesgue inequality and the isominwidth inequality only exists on the level of reduced convex bodies, we will have to control the the distance of a convex body to its reduction.

\begin{lemma}
\label{lemma:red_dist}
Let $K\subset S^2$ a convex body of width $w>\frac{\pi}{2}$ with $\area\left(K\right)\leq\area\left(U^{\circ}_{\pi-w}\right)+\varepsilon$ for some $\varepsilon>0$. Moreover, let $L$ be a reduction of $K$. Then,
$$
\delta\left(K,L\right)\leq\cc_w\sqrt{\varepsilon}
$$
for some positive constant $\cc_w$.
\end{lemma}
\proof
Let us assume that $L\subsetneq K$, otherwise there is nothing to prove. Let $x\in K\setminus L$. Let $\alpha$ be the breadth of the supporting lune $S$ to $L$ whose corners are $x$ and $-x$. Then $\alpha\geq w$. We observe that for $B=B\left(x,\frac{\delta\left(K,L\right)}{2}\right)$, we have $B\cap L=\emptyset$ and $B\cap S\subset K$ by convexity. Hence, from Theorem~\ref{thm:largewidth}, we get
$$
\area\left(U_{\pi-w}^\circ\right)\leq\area\left(L\right)\leq\area\left(K\right)-\area\left(B\cap S\right)\leq\area\left(U_{\pi-w}^\circ\right)+\varepsilon-\area\left(B\cap S\right),
$$
so we derive $\area\left(B\cap S\right)\leq\varepsilon$. On the other hand,
$$
\area\left(B\cap S\right)=\frac{\alpha}{2\pi}\area\left(B\right)\geq\frac{w}{2\pi}\area\left(B\right)=w\left(1-\cos\frac{\delta\left(K,L\right)}{2}\right)\geq\cc_w\delta\left(K,L\right)^2,
$$
and that concludes the proof.\endproof

Finally, we require the following lemma due to Glasauer \cite[Hilfssatz 2.2]{Gla95} in order to exploit the duality.

\begin{lemma}\label{lemma:Hausdorff:polarity}
Let $K,L\subset S^2$ be convex bodies such that $\delta\left(K,L\right)<\frac{\pi}{2}$. Then,
$$
\delta\left(K,L\right)=\delta\left(K^\circ,L^\circ\right).
$$
\end{lemma}

Now we are ready to prove our stability theorem for obtuse width. As in Section \ref{subsec:stab:smallw} it is enough to prove the theorem for small $\varepsilon$, since the Hausdorff distance of convex bodies in $S^n$ is bounded by $\pi$.

\vspace{6pt}\noindent{\textbf{Theorem~\ref{intro:thm:largewidthstab}a. }}\emph{
Let $U_{\pi-w}\subset S^2$ denote a Reuleaux triangle of width $\pi-w$ for $w\in\left(\frac{\pi}{2},\pi\right)$. There exist constants $\cc_w,\varepsilon_w>0$ depending only on $w$ such that for any $\varepsilon\in [0,\varepsilon_w)$ and any convex body $K\subset S^2$ with
    $\area(K)\leq\area(U_{\pi-w}^\circ) + \varepsilon$, we have $\delta(K,U^\circ)\leq\cc_w\sqrt\varepsilon$ for a certain Reuleaux triangle $U$ of width $\pi-w$.
}
\proof
In view of Lemma \ref{lemma:red_dist} we can assume that $K$ is already reduced.
We write $P=U_{\pi-w}^\circ$. The polar of a convex body of constant width on the sphere is again a body of constant width (cf.\ Proposition \ref{prop:constwidthpolar}), therefore we have the following due to \eqref{eq:area:polarity}:
\begin{equation}\label{eq:area:polarReuleaux}
\area\left(P\right)=2\pi-\left(2\pi-\area\left(U_{\pi-w}\right)\right)\cdot\tan\left(\frac{\pi-w}{2}\right).
\end{equation}
Similarly, since $K$ is also of constant width, the area of $K$ is
\begin{equation}\label{eq:area:K}
\area\left(K\right)=2\pi-\left(2\pi-\area\left(K^\circ\right)\right)\cdot\tan\left(\frac{\pi-w}{2}\right).
\end{equation}
By the assumption $\area\left(K\right)\leq \area\left(P\right)+\varepsilon$, using \eqref{eq:area:polarReuleaux} and \eqref{eq:area:K}, we can derive
$$
\area\left(K^\circ\right)\leq\area\left(U_{\pi-w}\right)+\cot\left(\frac{\pi-w}{2}\right)\varepsilon.
$$
As $K^\circ$ is a body of constant width $\pi-w<\frac{\pi}{2}$, we can use Theorem~\ref{BL-stab}. Hence,
$$
\delta\left(K^\circ,U\right)\leq\cc_w\varepsilon
$$
for some positive constant $\cc_w$ (depending only on $w$) and a Reuleaux triangle $U$ of width $\pi-w$. If $\varepsilon<\varepsilon_w\coloneq \cc_w^{-1}\tfrac \pi 2$, we can apply Lemma \ref{lemma:Hausdorff:polarity} and thus conclude the proof.
\endproof

One cannot hope for a bound of the form $\delta(K,U^\circ) \leq \cc_w \varepsilon$ in Theorem \ref{intro:thm:largewidthstab}. In fact, the right hand side has to be of order $\varepsilon^{2/3}$ or lower. To construct our sequence of examples, we fix a width $w>\tfrac \pi 2$ and $\varepsilon>0$. Let $U_{\pi-w}^\circ$ be a particular polar Reuleaux triangle of width $w$. By Proposition \ref{prop:uwcirc} we can express $U_{\pi-w}^\circ$ as the convex hull of three disks $B_1$, $B_2$ and $B_3$ of radius $w-\tfrac \pi 2$ centered at the vertices of a regular triangle of side length $\pi-w$. Hence, the boundary of $U_{\pi-w}^\circ$ is composed of three circular arcs and three geodesic segments. Let $x_1$ denote the center of $B_1$ and let $m_1$ denote the midpoint of the circular arc corresponding to $B_1$. Let $y$ be a point of distance $\varepsilon$ to $m_1$ such that $m_1\in [x_1,y]$. Then we define \(K_\varepsilon = \conv(U_{\pi-w}^\circ\cup\{y\}) \).

\begin{prop}
    \label{prop:23}
    In the above notation, the following statements hold for sufficiently small $\varepsilon>0$:
    \begin{enumerate}
        \item $w(K_\varepsilon) = w$,
        \item $\delta(K_\varepsilon,U^\circ) \geq \cc_w\varepsilon$ for all Reuleaux triangles $U$ of width $w$.
        \item $\area(K_\varepsilon) - \area(U_{\pi-w}^\circ) \leq \cc_w\varepsilon^{3/2}$
    \end{enumerate}
\end{prop}

\begin{proof}
    For the first statement we recall that the diameter of $U_{\pi-w}$ is attained by any segment that joins a vertex of $U_{\pi-w}$ with its opposite arc. Under the polarity operation, these diameter realizing segments correspond to width realizing lunes of $U_{\pi-w}^\circ$ (cf.\ Lemma \ref{lemma:luneslines}) that support $U_{\pi-w}^\circ$ at one of the segments in its boundary and the opposite circular arc. Thus, there is a width realizing lune of $U_{\pi-w}^\circ$ that contains $y$ in its interior (since $\varepsilon$ is small). It follows that $w(K_\varepsilon)=w(U_{\pi-w}^\circ)$ since $U_{\pi-w}^\circ$ is a subset of $K_\varepsilon$.

    For the second statement let $\ell$ be the supporting line to $U_{\pi-w}$ at any of its three vertices $v$ such that $\ell$ has the same angle with both of the adjacent circular arcs at $v$ in $\partial U_{\pi-w}$. Since $U_{\pi-w}$ is of constant width, there exists a width realizing lune $L$ of $U_{\pi-w}$ that uses $\ell$. By symmetry, the second geodesic line that defines $L$ is the supporting line of $U_{\pi-w}$ at the midpoint of the circular arc in $\partial U_{\pi-w}^\circ$ opposite to $v$. 
    
    Under polarity, this lune corresponds to a diameter realizing segment $S$ of $U_{\pi-w}^\circ$ (cf.\ Lemma \ref{lemma:luneslines}) that connects the midpoint of one of the circular arcs to the midpoint of the geodesic segment opposite to that arc. If we choose the circular arc to be the one contained in $B_1$ then, by symmetry, the center $x_1$ of $B_1$ is contained in $S$. Hence, $\diam(K_\varepsilon) \geq \diam(U_{\pi-w}^\circ) + \varepsilon$. The claim now follows through the same reasoning as in the proof of Proposition \ref{prop:optimality}.

    For the last statement, using that $\varepsilon>0$ is sufficiently small, it is enough to compute $\area(\conv(C) - \area(B_1)$, where $C = \conv (B_1 \cup\{y\})$.
    Let $z,z'\in\partial B_1$ be the points for which the the segments $[y,z]$ and $[y,z']$ are  contained in $\partial C$. Consider the triangle $T = [x_1,y,z]$. 
    This triangle has a right angle at $z$ and side length $r:=\pi-w,~ r+\varepsilon$ and $s(\varepsilon):=\arccos(\cos(r+\varepsilon)/\cos(r))$. Let $\alpha$ be its angle at $x_1$ and $\beta$ its angle at $y$. Then,
    \[
     \alpha(\varepsilon) = \arcsin\left(\frac{\sin(s(\varepsilon))}{\sin(r+\varepsilon)}   \right)\quad\text{and}\quad \beta(\varepsilon) = \arcsin\left(\frac{\sin(r)}{\sin(r+\varepsilon)}   \right).
    \]
    It follows that
    \[\begin{split}
    \area(C) - \area(B_1) &= 2(\area(T) - \area(T\cap B_1)) = 2(\alpha(\varepsilon) + \beta(\varepsilon) - \tfrac \pi 2 - \alpha(\varepsilon) (1-\cos(r)))\\
    &=2(\cos(r) \alpha(\varepsilon) + \beta(\varepsilon) - \tfrac \pi 2)=:f(\varepsilon).
    \end{split}\]
    The function $f(\varepsilon)$ is not differentiable at $\varepsilon=0$, but the function $g(\varepsilon)=f(\varepsilon^2)$ is differentiable, and computing a taylor approximation of $g$ around $\varepsilon=0$ we find
    \[
        g(\varepsilon) \leq \frac{1}{3\sqrt 2 }\sqrt{\tan(r)}\varepsilon^3 + \cc_w'\varepsilon^4.
    \]
    This implies $f(\varepsilon) \leq \cc_w \varepsilon^{3/2} $ as desired.
\end{proof}

\end{document}